\documentclass[10 pt,a4paper]{article}

\usepackage{amsfonts,amsmath,amsthm,amssymb,graphicx,fancyhdr,makeidx,amscd}
\usepackage{mathrsfs}
\usepackage{graphics}
\usepackage{enumitem, marvosym, mathtools}

\usepackage[T1]{fontenc}
\usepackage[utf8]{inputenc}
\usepackage{authblk}

\usepackage[all]{xy}

\newtheorem{thm}{Theorem}[section]
\newtheorem{lem}[thm]{Lemma}
\newtheorem{pro}[thm]{Proposition}
\newtheorem{cor}[thm]{Corollary}

\newtheorem{rmk}[thm]{Remark}

\numberwithin{equation}{section}

\usepackage{amssymb,amsmath}

\newcommand{\R}{\mathbb{R}}
\newcommand{\N}{\mathbb{N}}

\newcommand{\ep}{\varepsilon}

\newcommand{\ke}[1]{\mathcal{I}_K(#1)}

\newcommand{\linfb}[2]{\left\| #1\right\|_{\mathrm{L}^{\infty}\left(B_{#2}\right)}}
\newcommand{\linf}[2]{\left\| #1\right\|_{\mathrm{L}^{\infty}\left(#2\right)}}
\newcommand{\linfr}[1]{\left\| #1\right\|_{\mathrm{L}^{\infty}\left(\R^n\right)}}

\newcommand{\calpha}[2]{\left\| #1\right\|_{\mathcal{C}^{1,\alpha}\left(B_{#2}\right)}}

\newcommand{\Cinfb}[1]{\mathcal{C}^{\infty}\left(B_{#1}\right)}
\newcommand{\Gb}[2]{\mathcal{G}^{#1}\left(B_{#2}\right)}
\newcommand{\Gev}[2]{\mathcal{G}^{#1}\left(#2\right)}

\newcommand{\Laps}{(-\Delta)^s}

\newcommand{\ga}{\gamma}
\newcommand{\gas}{\gamma_*}
\newcommand{\bs}{\backslash}

\def\beq{\begin{equation}}
\def\eeq{\end{equation}}
\def\beqs{\begin{equation*}}
\def\eeqs{\end{equation*}}

\begin{document}

\title{Gevrey regularity for integro-differential operators}

\author[1]{Guglielmo Albanese\thanks{guglielmo.albanese@unimi.it}}
\author[2]{Alessio Fiscella \thanks{fiscella@ime.unicamp.br}}
\author[1,3,4]{Enrico Valdinoci\thanks{Enrico.Valdinoci@wias-berlin.de}}

\affil[1]{\small Dipartimento di Matematica \emph{Federigo Enriques}, Universit\`{a} degli Studi di Milano, Via~Saldini~50, I-20133, Milano (Italy)}
\affil[2]{Departamento de Matem\'atica, Universidade Estadual de Campinas, IMECC,\\ Rua S\'ergio Buarque de Holanda, 651, SP CEP 13083-859, Campinas, Brazil}
\affil[3]{\small Weierstra{\ss} Institut f\"ur Angewandte Analysis und Stochastik,\\ Mohrenstra{\ss}e 39, D-10117, Berlin Germany}
\affil[4]{\small Istituto di Matematica Applicata e Tecnologie Informatiche, CNR,\\ Via Ferrata 1, I-27100, Pavia, Italy}

\renewcommand\Authands{, and }

\maketitle

\begin{abstract}
We prove for some singular kernels $K(x,y)$ that viscosity solutions of the integro-differential equation
	\beqs
	\int_{\R^n} \left[u(x+y)+u(x-y)-2u(x)\right]\,K(x,y)dy=f(x)
	\eeqs
locally belong to some Gevrey class if so does $f$. The fractional Laplacian equation is included in this framework as a special case.
\end{abstract}

\maketitle

\section{Introduction}

Recently, a great attention has been devoted to equations driven
by nonlocal operators of fractional type. {F}rom
the physical point of view, these
equations take into account long-range particle interactions 
with a power-law decay.
When the decay at infinity is sufficiently weak,
the long-range phenomena may prevail
and the nonlocal effects persist even on large scales
(see e.g. \cite{CaffaSou, PSV, Ovidiu}).

The probabilistic counterpart of these fractional equations
is that the underlying diffusion
is driven by a stochastic process with power-law tail probability
distribution (the so-called Pareto or L\'evy distribution),
see for instance \cite{1Viswana, walk}.
Since long relocations are allowed by the process,
the diffusion obtained is sometimes referred to with the name of
anomalous (in contrast with the classical one coming from
Poisson distributions). Physical realizations of these models
occur in different fields, such as fluid dynamics (and especially
quasi-geostrophic and water wave equations), dynamical
systems, elasticity and micelles, see, among the others
\cite{Solo93, Cordoba, Craig, 1Shlesi}. Also, the scale invariance of
the nonlocal probability distribution may combine with
the intermittency and renormalization properties of
other nonlinear dynamics and produce complex patterns
with fractional features. For instance, there are
indications that the distribution of food
on the ocean surface has scale invariant properties (see
see e.g. \cite{VISW} and references therein) and it is possible
that optimal searches of predators reflect these patterns
in the effort of locating abundant food in sparse environments,
also considering that power-law distribution of movements allow the
individuals to visit more sites than the classical Brownian situation
(see e.g. \cite{3Bart, HUMP}).

The regularity theory of integro-differential equations has been
extensively studied in continuous and smooth spaces, see e.g.
\cite{SilvTH, SilvANN, BFV, DCKP}. The purpose of this paper
is to deal with the regularity theory in a Gevrey
framework. The proof combines a quantitative bootstrap
argument developed in \cite{BFV}
and the classical iteration scheme of \cite{MN, Mo}. Here the bootstrap
argument is more delicate than in the classical case
due to the nonlocality of the operator, since the value of
the function in a small ball is affected by the values of the
function everywhere, not only in a slightly bigger ball;
in particular the derivatives of the function cannot be controlled
in the whole space and a suitable truncation argument is needed.

Before stating the main results of the paper, we recall the definition of Gevrey function. For a detailed treatment of the theory of Gevrey functions and their relation with analytic functions we refer to \cite{KP, Ro}.
Let $\Omega\in\R^n$ be an open set, we define for any fixed real number $\sigma\geq 1$ the class $\mathcal{G}^{\sigma}(\Omega)$ of Gevrey functions of order $\sigma$ in $\Omega$. This is the set of functions $f\in\mathcal{C}^{\infty}(\Omega)$ such that for every compact subset $\Theta$ of $\Omega$ there exist positive constants $V$ and $\Gamma$ such that for all $i\in\N$
		\beqs
		\linf{D^if}{\Theta}\leq V\, \Gamma^i\, (i!)^{\sigma}\, .
		\eeqs
We remark that the spaces $\Gev{\sigma}{\Omega}$ form a nested family, in the sense that $\Gev{\sigma}{\Omega}\subseteq\Gev{\tau}{\Omega}$ whenever $\sigma\leq\tau$ and furthemore the inclusion is strict whenever the inequality is. Clearly the class $\Gev{1}{\Omega}$ coincides with $\mathcal{C}^{\omega}(\Omega)$, that of analytical functions. It should be stressed that both the inclusions
	\beqs
	\begin{aligned}
	\mathcal{C}^{\omega}(\Omega)\subset\bigcap_{\sigma>1}\Gev{\sigma}{\Omega} & \quad\quad\quad & \bigcup_{\sigma\geq 1}\Gev{\sigma}{\Omega}\subset\mathcal{C}^{\infty}(\Omega)
	\end{aligned}
	\eeqs
are strict, see \cite{Ro}. 
The notion of Gevrey class of functions is quite useful in applications. For instance, it possess a nice characterization in Fourier spaces.
Moreover, cut-off functions are never analytic, but they may be chosen to belong to a Gevrey space. Roughly speaking, for a smooth function $f$ the notion of Gevrey order measures "how much" the Taylor series of $f$ diverges.

As in \cite{BFV} we consider a quite general kernel $K=K(x,y):\R^n\times\left(\R^n\backslash\left\{0\right\}\right)\rightarrow\left(0,+\infty\right)$ satisfying some structural assumptions. From now we assume that $s\in\left(1\slash 2, 1\right)$.

We suppose that $K$ is close to the kernel of the fractional Laplacian in the sense that
	\beq\label{K1}
	\begin{cases}
	\hbox{there exist $a_0, r_0$ and $\eta\in\left(0,a_0\slash 4\right)$ such that}\\
	\displaystyle\left|\frac{|y|^{n+2s}K(x,y)}{2-2s}-a_0\right|\leq\eta\quad\quad \hbox{for all $x\in B_1$, $y\in B_{r_0}\backslash\left\{0\right\}$}.
	\end{cases}
	\eeq
Since we are interested in the Gevrey regularity, in order to ensure that our solution are $\mathcal{C}^{\infty}$ we assume that $K\in\mathcal{C}^{\infty}\left(B_1\times\left(\R^n\backslash\left\{0\right\}\right)\right)$ and moreover
	\beq\label{K2}
	\begin{cases}
	\hbox{for all $k\in\N\cup\left\{0\right\}$ there exist $\mathrm{H}_k>0$ such that}\\
	\displaystyle\linfb{D_x^{\mu}D_y^{\theta}K(\cdot,y)}{1}\leq\frac{\mathrm{H}_k}{\left|y\right|^{n+2s+|\theta|}}\\
	\quad\quad\quad \hbox{for all $\mu, \theta\in\N^n$, $|\mu|+|\theta|=k$, $y\in B_{r_0}\backslash\left\{0\right\}$}.
	\end{cases}
	\eeq
Furthermore, since we need a quantitative asymptotic control on the tails of the derivatives of $K$, we assume that
	\beq\label{K3}
	\begin{cases}
	\hbox{there exist $\nu\geq 0$ and $\Lambda>0$ such that}\\
	\displaystyle\mathrm{H}_k\leq \Lambda^k\left(k!\right)^{\nu}\quad\quad \hbox{for all $k\in\N\cup\left\{0\right\}$}.
	\end{cases}
	\eeq

We adopt the following notation for the second increment
	\beqs
	\delta u(x,y)=u(x+y)+u(x-y)-2u(x).
	\eeqs

Our main result is about the Gevrey regularity of the integro-differential operators with a general kernel $K$.

\begin{thm}\label{A}
Suppose $f$ is in $\Gb{\tau}{6}$ and suppose that $u\in \mathrm{L}^{\infty}(\R^n)$ is a viscosity solution of 
	\beq\label{equ}
		\int_{\R^n} \delta u(x,y)\,K(x,y)dy=f(x) \quad\quad \hbox{inside $B_6$}
	\eeq
with $s\in\left({1}\slash{2},1\right)$. Assume that $K:B_1\times\left(\R^n\backslash\left\{0\right\}\right)\rightarrow (0,+\infty)$ satisfies assumptions (\ref{K1}), (\ref{K2}), and (\ref{K3}) for $\eta$ sufficiently small.\\ 
Then there exists a $0<\overline{R}<5$ which depends only on $n$, $f$, and $\linfr{u}$ such that $u\in\Gb{\sigma}{R}$ for each $\sigma\geq\max\{1+\nu,\tau\}$ and $R\leq\overline{R}$.
\end{thm}
\begin{rmk}
We stress the fact that, as in \cite{BFV}, the kernel $K$ is not assumed to be symmetric in his entries, in particular, the cases considered here do not come necessarily from an extension problem, since they comprise rather general singular kernels. A delicate point is also the range of the constant $s$ appearing in the exponents, indeed in the paper it is assumed to belong to the interval $(1\slash 2,1)$ in order to exploit Theorem 5 of \cite{BFV} (in last analysis that range for $s$ descends from Theorem 61 of \cite{CS1}). It would be very interesting to develop a regularity theory also in the fractional case $s\in (0,1/2]$. Of course, the additional difficulty in this case is that the elliptic operator, at a first step, cannot regularize \emph{more than $2s$ derivatives}, and so not even the first derivative is under control. It is possible that some careful, intermediate H\"older bootstrap argument is needed to overcome this lack of initial regularity. We think that it would be interesting to analyze the special case of fractional Laplacian kernel where also the Caffarelli-Silvestre extension \cite{CSex} is at hand. It is an interesting problem to consider the regularity theory of boundary reactions of fractional type (with possible singular or degenerate elliptic operator), see e.g. \cite{SiV}.
\end{rmk} 
\begin{rmk}
It is important to stress the fact that, although Theorem \ref{A} seems to be purely local, the nonlocality of the operator gives also a result for the equation on a domain with a Dirichlet boundary condition. Indeed, if we consider a bounded solution of (\ref{equ}) on a bounded, smooth domain $\Omega$ such that $u=g$ on $\mathbb{R}^n\setminus\Omega$ with $g\in L^{\infty}(\mathbb{R}^n\setminus\Omega)$, then $u\in L^{\infty}(\mathbb{R}^n)$ and thus, for each point $x\in\Omega$ there exists a $r=r(x)>0$ such that we can apply the theorem to the ball of radius $r$ centered at $x$. It would be interesting to investigate whether the case of Dirichlet boundary conditions lead to a better regularity theory than the one in Theorem \ref{A}.
\end{rmk}
Furthermore we specialize the analysis to the case of the fractional Laplacian kernel and obtain the following

\begin{cor}\label{B}
Suppose $f$ is in $\Gb{\tau}{6}$ and suppose that $u\in \mathrm{L}^{\infty}(\R^n)$ is a viscosity solution of 
	\beq\label{equ2}
		\Laps u(x)=f(x) \quad\quad \hbox{inside $B_6$}
	\eeq
with $s\in\left({1}\slash{2},1\right)$.\\ 
Then there exists a $0<\overline{R}<5$ which depends only on $n$, $f$, and $\linfr{u}$ such that $u\in\Gb{\sigma}{R}$ for each $\sigma\geq\max\{2,\tau\}$ and $R\leq\overline{R}$.
\end{cor}

We remark that most of the difficulties in our
problem come from having the equation in a domain
rather than in the whole of the space and from
the fact that we look
for local, rather than global, estimates.
For instance, if a summable $u$ satisfies
\beqs
(-\Delta)^su(x)+f(x)=0
\eeqs 
for every $x\in\R^n$
and $f$ belongs to some Gevrey class, then so is $u$, and this can be
proved directly via Fourier methods, see e.g Theorem 1.6.1 in \cite{Ro}. In concrete cases, for instance, when the equation is set in a ball, the solution can be explicitly found by integration against a suitable, analytic kernel (see e.g. \cite{Bu}) and so the regularity theory can follow from the explicit representation of the solution.

Concerning the Gevrey exponent~$\sigma=\max\{1+\nu,\tau\}$
in Theorem \ref{A}, we think that it is an interesting open problem
to establish whether or not such exponent is optimal or it can
be lowered, for instance, to the value~$\max\{\nu,\tau\}$
(in our case, the exponent value~$\sigma$
is due to a nonlocal boundary effect
and the worst factor comes from the higher
derivatives of the kernel).
\begin{rmk}
We think that an interesting problem is to determine also a Gevrey regularity theory for the powers of the Laplacian in the spectral sense (see for instance the very recent \cite{CaSt, Gr}). Smooth and analytic regularity theories in this framework have been previously considered in \cite{KN}.
\end{rmk}

\section{Incremental quotients}
In this section we recall some basic facts and results about incremental quotients. 

Given $k\in\N$,
we observe that, for any $a\in \R$,
there are exactly $k+1$ integers belonging to the interval $(a,a+k+1]$
(and, moreover, they are consecutive,
this can be easily proved by induction over $k$).
In particular, taking $a:=-(k+1)/2$, we have that
there are exactly $k+1$ integers belonging to the interval $(
-(k+1)/2 ,(k+1)/2]$. We call these integers $j_1,\dots,j_{k+1}$.

Now, we consider $V\in \,{\rm Mat}\,((k+1)\times(k+1))$
to be the matrix
$$ V_{m,i}:= {j_i}^{m-1}
{\mbox{ for any }}i ,m=1,\dots,k+1.$$
Then, we consider the vector $c^{(k)} 
=(c^{(k)}_1,\dots,c^{(k)}_{k+1})\in\R^{k+1}$
as the unique solution of
\beq\label{Van}
Vc=(0,\dots,0,k!).
\eeq
We remark that $V$ is invertible, being a Vandermonde matrix,
and therefore the definition of $c^{(k)}$ is well posed.

With this notation,
given $h>0$, $v \in {\mathbb S}^{n-1}$, and a function 
$u$, we consider the $h$-incremental quotient of $u$
of order $k$ in direction $v$, that is
\beq\label{not h}
T_h^v u(x):= \sum_{i=1}^{k+1} c^{(k)}_{i} u(x+j_i h 
v).
\eeq
We remark that
$ T_h^v u(x) / h^k$ behaves like
${D^k_v u(x)}$, as next result shows:

\begin{lem}\label{I}
Let $r>(k+1) h>0$ and $u\in \mathcal{C}^k (B_r( x))$. Then
$$ T_h^v u(x) = h^k {D^k_v u(x)} + o(h^{k}).$$  
\end{lem}

\begin{proof} By Taylor's Theorem
$$ u(x+j_i h v) = \sum_{\ell=0}^{k-1} \frac{D^\ell_v 
u(x)}{\ell!}j_i^\ell h^\ell +
\frac{D^k_v 
u(x+\xi_i)}{k!}j_i^k h^k$$
for some $\xi_i\in\R^n$ with $|\xi_i|\le (k+1)h$.
Then we have
\begin{eqnarray*}
T_h^v u(x) &=& \sum_{i=1}^{k+1} 
\sum_{\ell=0}^{k-1} \frac{D^\ell_v
u(x)}{\ell!} c^{(k)}_{i} j_i^\ell h^\ell +
\sum_{i=1}^{k+1} c^{(k)}_{i}
\frac{D^k_v u(x+\xi_i)}{k!}j_i^k h^k
\\ &=& 
\sum_{i=1}^{k+1} 
\sum_{m=1}^{k} \frac{D^\ell_v
u(x)}{\ell!} c^{(k)}_{i} V_{m,i} h^{m-1} +
\sum_{i=1}^{k+1} c^{(k)}_{i}
\frac{D^k_v u(x+\xi_i)}{k!} V_{k+1,i} h^k
\\ &=&
\sum_{m=1}^{k} \frac{D^\ell_v
u(x)}{\ell!} (V c^{(k)})_m h^{m-1} +
(V c^{(k)})_{k+1}
\frac{D^k_v u(x)}{k!} h^k
\\ &&\quad+\sum_{i=1}^{k+1} c^{(k)}_{i}
\frac{D^k_v u(x+\xi_i)-D^k_v u(x)}{k!} V_{k+1,i} h^k
\end{eqnarray*}
This and \eqref{Van} imply 
\begin{equation}\label{remainder}
T_h^v u(x) =h^k
{D^k_v u(x)}
+h^k\sum_{i=1}^{k+1} c^{(k)}_{i}
\frac{D^k_v u(x+\xi_i)-D^k_v u(x)}{k!} V_{k+1,i}
\end{equation}   
and this gives
the desired results.
\end{proof}

We observe that $T_h^v(T_h^{\tilde v} u)=
T_h^{\tilde v}(T_h^vu)$, hence we can extend
the notation in \eqref{not h} to the multi-index case.
Namely, if $k=(k_1,\dots, k_\ell)\in \N^\ell$
and $v=(v_1,\dots,v_\ell)\in ({ \mathbb{S}}^{n-1})^\ell$, we define
\begin{eqnarray*}
&& T_h^v u(x):= T_h^{v_1}(\dots (T_h^{v_\ell} 
u(x))\dots)
\\ &&\quad=\sum_{i_1=1}^{k_1+1}\dots\sum_{i_\ell=1}^{k_\ell+1}
c^{(k_1)}_{i_1}\dots c^{(k_\ell)}_{i_\ell}
u(x+j_{i_1} h v_1+\dots+j_{i_\ell} h v_\ell)\end{eqnarray*}
and this definition is, in fact, independent of the order.
In this way, 
we have the following multi-index version of Lemma \ref{I}:

\begin{lem}\label{incr}
Let $r>(k+1) h>0$ and $u\in \mathcal{C}^k (B_r( x))$. Then
$$ T_h^v u(x) = h^{|k|} {D^k_v u(x)} + o(h^{|k|}).$$
\end{lem} 

To prove Theorem \ref{A} we will need the following integral analogue of Lemma \ref{incr}. From now on, we will adopt the convention that for any kernel $K(x,y)$
	\beqs
	T_{h}^{\ga}K(x,y)
	\eeqs
denotes the incremental quotient with respect to $y$ (i.e. letting fixed $x$).
	\begin{pro}\label{proint}
	Let $0<\ep<r$ and $K\in\mathcal{C}^{\infty}(B_1\times\R^n\bs B_{r-\ep})$ satisfying condition (\ref{K2}).
	Then for each $\ga\in\N^n$
		\beqs
		\lim_{h\rightarrow 0}\underset{r<|y|}{\int}\frac{1}{h^{|\ga|}}\left|T_{h}^{\ga}K(x,y)\right|dy\,=\,\underset{r<|y|}{\int}\left|D_y^{\ga}K(x,y)\right|dy\, .
		\eeqs 
	\end{pro}
		\begin{proof}
		For $\ep>h(|\ga|+1)>0$ we set 
		\beqs
		f_h(x,y)=\frac{1}{h^{|\ga|}}\left|T_h^{\ga}K(x,y)\right|\, ,
		\eeqs
		it follows from Lemma \ref{I} that $f_h(x,y)$ converges pointwise in $\Omega$ to 
		\beqs
		f(x,y)=\left|D_y^{\ga}K(x,y)\right|\, .
		\eeqs 
		The idea is to show that there exists a $g\in L^1(\Omega)$ such that $f_h(y)\leq g(y)$. From (\ref{remainder}) it follows that there exists $\xi_i\in\R^n$ with $|\xi_i|\leq h(|\ga|+1)$ and positive constants $A$ and $A_i$ such that
			\beqs
				\begin{aligned}
				f_h(x,y) & \leq |D_y^{\ga} K(x,y)| +\sum_{i=1}^{|\ga|+1} c^{(|\ga|)}_{i} \frac{\left|D_y^{\ga} K(x,y+\xi_i)-D_y^{\ga} K(x,y)\right|}{|\ga|!} V_{|\ga|+1,i}\\
				&\leq A \left|D_y^{\ga} K(x,y)\right| + \sum_{i=1}^{|\ga|+1} A_i \left|D_y^{\ga} K(x,y+\xi_i)\right|\, .
				\end{aligned}
			\eeqs
		From condition (\ref{K2}) and $\ep<r$ we have 
			\beqs
			\begin{aligned}
			\left|D_y^{\ga} K(x,y+\xi_i)\right| & \leq \frac{\mathrm{H}_{\left|\ga\right|}}{\left|y+\xi_i\right|^{n+2s+\left|\ga\right|}}\\
			& \leq \frac{\mathrm{H}_{\left|\ga\right|}}{\left(\left|y\right|-\left|\xi_i\right|\right)^{n+2s+\left|\ga\right|}}\\
			& \leq \frac{\mathrm{H}_{\left|\ga\right|}}{\left(\left|y\right|-\ep\right)^{n+2s+\left|\ga\right|}}\, .	
			\end{aligned}
			\eeqs
		
		Since this last estimate is independent of $h$, the proposition follows from the Lebesgue dominated convergence theorem.
		\end{proof} 

\section{An a priori estimate}
In this section we deduce the following \emph{a priori} estimate of Friedrichs type for solutions of (\ref{equ}). It will be the key tool in proving the Gevrey regularity of the solutions of the equation.

\begin{lem}\label{hard}
Let $u$ be a solution of (\ref{equ}) with $s\in\left(\frac{1}{2},1\right)$ and $f\in\mathcal{C}^{\infty}(B_6)$. Then for any $0<r<r+\delta<5$ and $p\geq 0$ the following estimate holds
	\beqs
		\begin{aligned}
		\linfb{\nabla^{p+2}u}{r} & \leq C\left[\frac{1}{\delta}\linfb{\nabla^{p+1}u}{r+\delta} + \frac{1}{\delta^2} \linfb{\nabla^{p}u}{r+\delta}\right.\\
		& \left. \quad\quad + \delta^{2s-1}\linfb{\nabla^{p+1}f}{r+\delta} + \frac{\mathrm{H}_{p+1}2^p}{\delta^{p+2}}\linfr{u} \right],
		\end{aligned}
	\eeqs
	where $C$ is a constant depending only on $n$ and $s$.
\end{lem}
We note that the estimate above corresponds to Lemma 5.7.1 in \cite{Mo}. It should be stressed that while in the local case the estimate for a generic $p$ is easily deduced from the estimate for $p=0$ by differentiating the equation, in the nonlocal case we have to be more accurate in order to take into account the long range interaction. In particular to obtain this estimate we will exploit the bootstrap machinery developed in the \cite{BFV} to prove the $\mathcal{C}^{\infty}$ regularity of solutions.
 
In the rest of the paper we will adopt the following notation in order to keep the exposition clear. For $\Omega\in\R^n$, $u$ a measurable function, and $K$ a kernel as in Theorem \ref{A}
	\beqs
	\ke{u}=\int_{\R^n} \delta u(x,y)\,K(x,y)dy\,.
	\eeqs 
\begin{proof}
First of all we recall that by Theorem 5 in \cite{BFV} viscosity solutions of (\ref{equ}) are of class $\mathcal{C}^{\infty}$. Now we choose a $\mathcal{C}^{\infty}$ cutoff function $\eta(x)\geq 0$ such that
	\beqs
	\eta(x)=
		\begin{cases}
		1 & B_{4}\\
		0 & \R^n\bs B_{5}\\
		\end{cases}
	\eeqs
and 
	\beq\label{Deta}
	\linfb{\nabla\eta(x)}{5}\leq 2\, .
	\eeq
Now for $\ga\in\N^n$ with $|\ga|=p$	and $e_i$ a unit vector we set $\gas=\ga+e_i$ and define for each $0<h<\frac{1}{p+2}$
	\beqs
	w(x)=\frac{1}{h^{p+1}}T^{\gas}_hu(x)\, .
	\eeqs
By linearity $w(x)=w_1(x)+w_2(x)$ where
	\beqs
	w_1(x)=\frac{1}{h^{p+1}}T_h^{e_i}\left(\eta\ T^{\ga}_hu(x)\right)\, ,
	\eeqs
and
	\beqs
	w_2(x)=\frac{1}{h^{p+1}}T_h^{e_i}\left((1-\eta)\ T^{\ga}_hu(x)\right).
	\eeqs
The strategy is to get a good estimate of $\left|\ke{w_1}\right|$ inside $B_1$ using the technique in \cite{BFV}. We stress the fact that $w_2(x)\equiv 0$ for $x\in B_1$ since $\eta\equiv 1$ there.\\
From the triangle inequality we have that
 
	\beq\label{uuu}
		\begin{aligned}
		\left|\ke{w_1}\right| & = \left|\ke{w} - \ke{w_2} \right|\\
		& \leq \left|\ke{w}\right|+\left|\ke{w_2}\right|\\
		& = \left|\frac{1}{h^{p+1}}T^{\gas}_hf(x)\right|+\left|\ke{w_2} \right|\, .
		\end{aligned}
	\eeq
Using the discrete integration by parts and recalling that $x\in B_1$ we get the inequality	
	\beqs
		\begin{aligned}
		\left|\ke{w_2}\right| & = \left|\int_{\R^n}\left[ w_2(x+y)+w_2(x-y)-2w_2(x)\right]\ K(x,y)\ dy \right|\\
		& = \left|\int_{\R^n}\left[ w_2(x+y)+w_2(x-y)\right]\ K(x,y)\ dy \right|\\
		& \leq \left|\int_{\R^n} w_2(x+y)\ K(x,y)\ dy \right| + \left|\int_{\R^n} w_2(x-y)\ K(x,y)\ dy \right|\\
		& = \left|\int_{\R^n}\left[ \frac{1}{h^{p+1}}T_h^{e_i}\left((1-\eta)\ T^{\ga}_hu(x+y)\right)\right]\ K(x,y)\ dy \right|\\
		& \quad\quad\quad+ \left|\int_{\R^n}\left[ \frac{1}{h^{p+1}}T_h^{e_i}\left((1-\eta)\ T^{\ga}_hu(x-y)\right)\right]\ K(x,y)\ dy \right|\\
		& = \left|\int_{\R^n}\left[ \frac{1}{h^{p+1}}\left((1-\eta)\ T^{\ga}_hu(x+y)\right)\right]\ T_{-h}^{e_i}K(x,y)\ dy \right|\\
		& \quad\quad\quad+ \left|\int_{\R^n}\left[ \frac{1}{h^{p+1}}\left((1-\eta)\ T^{\ga}_hu(x-y)\right)\right]\ T_{-h}^{e_i}K(x,y)\ dy \right|\, .
		\end{aligned}
	\eeqs
Now, exploiting the properties of the cutoff function $\eta$ we get the following estimate
	\beqs
		\begin{aligned}
		\left|\ke{w_2}\right| & \leq \left|\underset{4<|x+y|<5}{\int}\left[ \frac{1}{h^{p+1}}\left((1-\eta)\ T^{\ga}_hu(x+y)\right)\right]\ T_{-h}^{e_i}K(x,y)\ dy \right|\\
		& \quad\quad\quad+ \left|\underset{4<|x-y|<5}{\int}\left[ \frac{1}{h^{p+1}}\left((1-\eta)\ T^{\ga}_hu(x-y)\right)\right]\ T_{-h}^{e_i}K(x,y)\ dy \right|\\
		& \quad\quad\quad\quad\quad\quad+ \left|\underset{5<|x+y|}{\int}\frac{1}{h^{p+1}}u(x+y)\ T_{-h}^{\gas}K(x,y)\ dy \right| \\
		& \quad\quad\quad\quad\quad\quad\quad\quad\quad+ \left|\underset{5<|x-y|}{\int}\frac{1}{h^{p+1}}u(x-y)\ T_{-h}^{\gas}K(x,y)\ dy \right|\, .
		\end{aligned}
	\eeqs
Using the triangle inequality we conclude that	
	\beqs
		\begin{aligned}		
		\left|\ke{w_2}\right| & \leq \left|\underset{\substack{2<|y| \\ |x+y|<5}}{\int} \left[ \frac{1}{h^{p+1}}\left((1-\eta)\ T^{\ga}_hu(x+y)\right)\right]\ T_{-h}^{e_i}K(x,y)\ dy \right|\\
		& \quad\quad\quad+ \left|\underset{\substack{2<|y| \\ |x-y|<5}}{\int}\left[ \frac{1}{h^{p+1}}\left((1-\eta)\ T^{\ga}_hu(x-y)\right)\right]\ T_{-h}^{e_i}K(x,y)\ dy \right|\\
		& \quad\quad\quad\quad\quad\quad+ \left|\underset{3<|y|}{\int}\frac{1}{h^{p+1}}u(x+y)\ T_{-h}^{\gas}K(x,y)\ dy \right|\\
		& \quad\quad\quad\quad\quad\quad+ \left|\underset{3<|y|}{\int}\frac{1}{h^{p+1}}u(x-y)\ T_{-h}^{\gas}K(x,y)\ dy \right|
		\end{aligned}
	\eeqs
and the RHS is less than or equal to
	\beqs
	2\linfb{\frac{1}{h^{p}}T^{\ga}_hu}{5}  \underset{2<|y|}{\int}\frac{1}{h}\ \left|T_{-h}^{e_i}K(x,y) \right|\ dy + 2\linfr{u} \underset{3<|y|}{\int}\frac{1}{h^{p+1}}\ \left|T_{-h}^{\gas}K(x,y)\right|\ dy\, .
	\eeqs
Inserting the inequality above into (\ref{uuu}) we obtain the desired estimate inside the ball of radius 1
	\beq\label{vvv}
	\begin{aligned}
	\left|\ke{w_1}\right| & \leq  \left|\frac{1}{h^{p+1}}T^{\gas}_hf(x)\right| + 2\linfr{u} \underset{3<|y|}{\int}\frac{1}{h^{p+1}}\ \left|T_{-h}^{\gas}K(x,y)\right|\ dy \\ 
	& \quad\quad\quad  + 2\linfb{\frac{1}{h^{p}}T^{\ga}_hu}{5}  \underset{2<|y|}{\int}\frac{1}{h}\ \left|T_{-h}^{e_i}K(x,y) \right|\ dy\, .
	\end{aligned}
	\eeq
Furthermore, thanks to the discrete Leibnitz rule we have that
	\beqs
	w_1(x)=\frac{1}{h^{p+1}}\left[\eta(x+he_i)T^{\gas}_hu(x)+T_h^{e_i}\eta(x)T_h^{\ga}u(x) \right]
	\eeqs
and this implies 
	\beq\label{w1inf}
	\linfr{w_1}\leq\linf{\frac{1}{h^{p+1}}T^{\gas}_hu}{B_5(he_i)}+\linf{\frac{1}{h}T_h^{e_i}\eta}{B_5}\linf{\frac{1}{h^{p}}T^{\ga}_hu}{B_5}
	\eeq
where $B_5(he_i)$ is the ball of radius $5$ centered at $he_i$.\\
Thanks to (\ref{vvv}), (\ref{w1inf}), and Theorem 61 in \cite{CS1} we get the following estimate for $w_1$
	\beqs
	\begin{aligned}
	\calpha{w_1}{1} & \leq C\left(\linfr{w_1}+\linfb{\ke{w_1}}{2}\right)\\
	& \leq C\left( \linf{w}{B_5(he_i)}+\linf{\frac{1}{h}T_h^{e_i}\eta}{B_5}\linf{\frac{1}{h^{p}}T^{\ga}_hu}{B_5}+\linfb{\frac{1}{h^{p+1}}T^{\gas}_hf}{2} \right. \\ 
	& \quad\quad\left.+ 2\linfb{\frac{1}{h^{p}}T^{\ga}_hu}{5}  \underset{2<|y|}{\int}\frac{1}{h}\ \left|T_{-h}^{e_i}K(x,y) \right|\ dy\right.\\ 
	& \quad\quad\left.+ 2\linfr{u} \underset{3<|y|}{\int}\frac{1}{h^{p+1}}\ \left|T_{-h}^{\gas}K(x,y)\right|\ dy \right)\, ,
	\end{aligned}
	\eeqs
where, from now on, $C$ will denote a positive constant depending only on $n$ and $s$. This means that in calculations we will reabsorb constant factors inside $C$. Letting $h\rightarrow 0$, from (\ref{K2}), Proposition \ref{proint}, and (\ref{Deta}) we conclude that
	\beq\label{est1}
	\begin{aligned}
		\calpha{D^{\gas}u}{1} & \leq C\left( \linfb{D^{\gas}u}{5}+2\linfb{D^{\ga}u}{5}+ \left\|D^{\gas}f\right\|_{\mathrm{L}^{\infty}(B_2)}\right. \\
		& \quad\quad \left.+ \mathrm{H}_1\linfb{D^{\ga}u}{5} + \frac{\mathrm{H}_{p+1}}{2^{2s+p+1}}\linfr{u} \right)\, .
	\end{aligned}
	\eeq
By a scaling argument, from (\ref{est1}) we get the following estimate valid for any $\sigma\in (0,1]$
	\beqs
		\begin{aligned}
		\linfb{\nabla^{p+2}u}{\sigma} & \leq C\left[\frac{1}{\sigma}\linfb{\nabla^{p+1}u}{5\sigma} +	\frac{1}{\sigma^2} \linfb{\nabla^{p}u}{5\sigma}\right.\\
		& \left. \quad\quad\quad + \sigma^{2s-1}\linfb{\nabla^{p+1}f}{2\sigma } + \frac{\mathrm{H}_{p+1}}{2^p}\frac{1}{\sigma^{p+2}}\linfr{u} \right]
		\end{aligned}
	\eeqs
this implies, by covering, the following global estimate, valid for $0<r<r+\delta\leq 5$
	\beqs
		\begin{aligned}
		\linfb{\nabla^{p+2}u}{r} & \leq C\left[\frac{1}{\delta}\linfb{\nabla^{p+1}u}{r+\delta}+ \frac{1}{\delta^2} \linfb{\nabla^{p}u}{r+\delta}\right.\\
		& \left. \quad\quad + \delta^{2s-1}\linfb{\nabla^{p+1}f}{r+\delta} + \frac{\mathrm{H}_{p+1}\,2^p}{\delta^{p+2}}\linfr{u} \right].
		\end{aligned}
	\eeqs
which is the desired estimate.
\end{proof}
	
\section{Proof of Theorem 1.1}
In this section we will adapt the iteration scheme in \cite{MN, Mo} to obtain Gevrey regularity of solutions of (\ref{equ}). The main idea here is to introduce some rescaled companions of the $\mathrm{L}^{\infty}$ norms in order to exploit effectively the estimate of Lemma \ref{hard}.\\

As in \cite{Mo, MN} we introduce the following.\\ 
For $f$ and $u$ in $\Cinfb{R}$, we define the quantities
	\beqs
		\begin{aligned}
		\mathrm{M}^s_{R,p}(f) & =\sup_{R\slash 2<r<R}(R-r)^{2s+p+1}\linfb{\nabla^{p+1}f}{r}, & & p\in\N\cup\left\{0\right\}\\
		\mathrm{N}^*_{R,p}(u) & =\sup_{R\slash 2<r<R}(R-r)^{p+2}\linfb{\nabla^{p+2}u}{r}, & & p\in\left\{-2,-1\right\}.
		\end{aligned}
	\eeqs
Since there is no ambiguity, in the rest of the paper will be suppressed the explicit dependence from $f$ and $u$ in $\mathrm{M}^s_{R,p}$ and $\mathrm{N}^*_{R,p}$.
It is apparent from this definition that a function $u \in \Cinfb{R}$ will be also in $\Gb{\sigma}{R}$ if and only if there exist positive constants $V$ and $\Gamma$ such that the following inequality holds for any $p\geq -2$
	\beq\label{key}
	\mathrm{N}^*_{R,p}\leq V\cdot \Gamma^{p}\cdot [p!]^{\sigma},
	\eeq
where $[p!]$ is defined as
	\beqs
		\begin{cases}
		p! & \hbox{if $p\geq0$}\\
		1 & \hbox{if $p<0$.}
		\end{cases}
	\eeqs
Thus the strategy for proving Theorem \ref{A} will consist in showing the validity of (\ref{key}) for solutions of (\ref{equ}). The following lemma will be the induction step in the proof.
\begin{lem}\label{step}
Let $u$ be a solution of \ (\ref{equ}) with $f\in\mathcal{C}^{\infty}(B_6)$. Then there exist positive constants $E$ and $F$ depending only on $s$ and $n$ such that the estimate
	\beqs
	\mathrm{N}^*_{R,p}\leq E \left[p\, \mathrm{N}^*_{R,p-1} + p(p-1)\, \mathrm{N}^*_{R,p-2} +   
	\mathrm{M}^s_{R,p} + F^p\,\mathrm{H}_{p+1}\,p!\linfr{u} \right]
	\eeqs
holds for any $p \geq 0$.
\end{lem}	
	\begin{proof}
	By Theorem 5 in \cite{BFV} we have that $u\in\Cinfb{R}$. Plugging the estimate of Lemma \ref{hard} into the definition of $\mathrm{N}^*_{R,p}$ we obtain the following
	\beq\label{N1}
		\begin{aligned}
		\mathrm{N}^*_{R,p} & \leq C \sup_{R\slash 2<r<R}(R-r)^{p+2} \left[\frac{1}{\delta}\linfb{\nabla^{p+1}u}{r+\delta} + \frac{1}{\delta^2} \linfb{\nabla^{p}u}{r+\delta}\right.\\
		& \left. \quad\quad\quad\quad\quad\quad + \delta^{2s-1}\linfb{\nabla^{p+1}f}{r+\delta} + \frac{\mathrm{H}_{p+1}2^p}{\delta^{p+2}}\linfr{u} \right].
		\end{aligned}
	\eeq
By the very definition of $\mathrm{M}^s_{R,p}$ and $\mathrm{N}^*_{R,p}$, we get the inequalities
	\beq\label{N2}
		\begin{aligned}
		\linfb{\nabla^{p+1}u}{r+\delta} & \leq \frac{1}{(R-r-\delta)^{p+1}}\mathrm{N}^*_{R,p-1}\\ 
		\linfb{\nabla^{p}u}{r+\delta} & \leq \frac{1}{(R-r-\delta)^{p}}\mathrm{N}^*_{R,p-2}\\ 
		\linfb{\nabla^{p+1}f}{r+\delta} & \leq \frac{1}{(R-r-\delta)^{2s+p+1}}\mathrm{M}^s_{R,p}
		\end{aligned}
	\eeq
and inserting them into (\ref{N1}) we have
		\beqs
		\begin{aligned}
		\mathrm{N}^*_{R,p} & \leq C \sup_{r\in[R\slash 2,R]} \left[\frac{(R-r)^{p+2}}{\delta (R-r-\delta)^{p+1}}\mathrm{N}^*_{R,p-1} + \frac{(R-r)^{p+2}}{\delta^2 (R-r-\delta)^p} \mathrm{N}^*_{R,p-2}\right.\\
		& \left. \quad\quad\quad\quad\quad\quad + \frac{\delta^{2s-1}(R-r)^{p+2}}{(R-r-\delta)^{2s+p+1}} \mathrm{M}^s_{R,p} + \mathrm{H}_{p+1}2^p \frac{(R-r)^{p+2}}{\delta^{p+2}}\linfr{u} \right]\, .
		\end{aligned}
		\eeqs
Now we eliminate the dependence on $r$ by setting
	\beqs
	\delta=\frac{R-r}{p},
	\eeqs		
in this way we get the following estimate
		\beqs
		\begin{aligned}
		\mathrm{N}^*_{R,p}  & \leq C \left[\frac{p^{p+2}}{(p-1)^{p+1}}\mathrm{N}^*_{R,p-1} + \frac{p^{p+2}}{(p-1)^p} \mathrm{N}^*_{R,p-2}\right.\\
		& \left. \quad\quad\quad\quad\quad\quad + \frac{p^{p+2}}{(p-1)^{2s+p+1}} \mathrm{M}^s_{R,p} + \mathrm{H}_{p+1}2^p p^{p+2}\linfr{u} \right]\\
		& \leq C \left[p\left(\frac{p}{p-1}\right)^{p+1} \mathrm{N}^*_{R,p-1} + p(p-1) \left(\frac{p}{p-1}\right)^{p+1}\mathrm{N}^*_{R,p-2}\right.\\
		& \left. \quad\quad\quad\quad\quad\quad + \frac{p}{(p-1)^{2s}} \left(\frac{p}{p-1}\right)^{p+1} \mathrm{M}^s_{R,p} + \mathrm{H}_{p+1}2^p p^{p+2}\linfr{u} \right]\\
		& \leq C \left\{\left(\frac{p}{p-1}\right)^{p+1}\left[p \mathrm{N}^*_{R,p-1} + p(p-1) \mathrm{N}^*_{R,p-2}+ \mathrm{M}^s_{R,p}\right]\right.\\
		&\quad\quad\quad\quad\quad\quad\left.+  \mathrm{H}_{p+1}2^p p^{p+2}\linfr{u}\right\} ,
		\end{aligned}
	\eeqs
where in the last inequality we have exploited the fact that $2s>1$. Now the proof follows since $\left(\frac{p}{p-1}\right)^{p+1}$ is a bounded quantity. 
\end{proof}

By Theorem 5 in \cite{BFV} we know that $u$ is of class $\mathcal{C}^{\infty}$ inside the ball $B_6$, this implies that the quantities $\mathrm{N}^*_{R,p}$ are well defined and finite for any $p\geq-2$ and $R<6$. Since $\Gb{\tau}{6}\subset\Cinfb{6}$ also $\mathrm{M}^s_{R,p}$ is well defined for all $p$, moreover there exist positive numbers $L$ and $A$ such that for any $R\leq 6$ the derivatives of $f$ satisfy
	\beqs
	\linfb{\nabla^p f}{R}\leq L\,\left(\frac{A}{R}\right)^p\,(p!)^{\tau}\, .
	\eeqs
This implies that
	\beqs
		\begin{aligned}
		\mathrm{M}^s_{R,p} & = \sup_{R\slash 2<r<R}(R-r)^{2s+p+1}\linfb{\nabla^{p+1}f}{r}\\
		& \leq L\,\left(\frac{R}{2}\right)^{2s}\,\left(\frac{A}{2}\right)^{p+1}\,((p+1)!)^{\tau}\, .
		\end{aligned}
	\eeqs
From Lemma \ref{step} it follows that
	\beq\label{N3}
	\begin{aligned}
	\mathrm{N}^*_{R,p} & \leq E \left[p\, \mathrm{N}^*_{R,p-1} + p(p-1)\, \mathrm{N}^*_{R,p-2} +   
	 F^p\, (p!)^{1+\nu}\linfr{u}\right.\\
	& \quad\quad\quad\quad\left.+ L\,\left(\frac{R}{2}\right)^{2s}\,\left(\frac{A}{2}\right)^{p+1}\,((p+1)!)^{\tau} \right].
	\end{aligned}
	\eeq
Theorem \ref{A} will be proved by showing that we can choose $\Gamma$ and $V$ so that (\ref{key}) holds. For this, we are going to proceed by induction on $p$. Clearly we can choose $\Gamma$ and $V$ so that (\ref{key}) holds for $p=-2$ and $p=-1$, then we suppose that this choice holds up to $p-1$ with $p>0$ and we prove it for $p$. From (\ref{N3}) we have that
	\beqs
	\begin{aligned}
	\mathrm{N}^*_{R,p} & \leq E \left[V\cdot \Gamma^{p-1}\cdot p [(p-1)!]^{\sigma} + V\cdot \Gamma^{p-2}\cdot  p(p-1) [(p-2)!]^{\sigma} \right.\\
	& \quad\quad\quad\quad\left. +	F^p\, (p!)^{1+\nu}\linfr{u} + L\,\left(\frac{R}{2}\right)^{2s}\,\left(\frac{A}{2}\right)^{p+1}\,((p+1)!)^{\tau} \right]\\
	& = V\cdot \Gamma^p\cdot  [p!]^{\sigma}\,E\left[\frac{1}{\Gamma}\, p^{1-\sigma} + \frac{1}{\Gamma^2}(p(p-1))^{1-\sigma} +	\frac{1}{V}\left(\frac{F}{\Gamma}\right)^p\, (p!)^{1+\nu-\sigma}\linfr{u} \right.\\
	& \quad\quad\quad\quad\left. + \frac{L}{V\,\Gamma^p}\,\left(\frac{R}{2}\right)^{2s}\,\left(\frac{A}{2}\right)^{p+1}\,(p+1)^{\tau}(p!)^{\tau-\sigma} \right]\, .\\
	\end{aligned}
	\eeqs
If we choose $\sigma\geq\max\{1+\nu,\tau\}$ the inequality above implies that
	\beqs
	\begin{aligned}
	\mathrm{N}^*_{R,p} &  = V\cdot \Gamma^p\cdot  [p!]^{\sigma}\,E\left[\frac{1}{\Gamma} + \frac{1}{\Gamma^2} +	\frac{1}{V}\left(\frac{F}{\Gamma}\right)^p\, \linfr{u} \right.\\
	& \quad\quad\quad\quad\left. + \frac{L}{V\,\Gamma^p}\,\left(\frac{R}{2}\right)^{2s}\,\left(\frac{A}{2}\right)^{p+1}\,(p+1)^{\tau} \right]\, .\\
	\end{aligned}
	\eeqs
At this point we are left to show that it is possible to choose $V$ and $\Gamma$ in such a way that
	\beq\label{N4}
	E\left[\frac{1}{\Gamma} + \frac{1}{\Gamma^2} +	\frac{1}{V}\left(\frac{F}{\Gamma}\right)^p\, \linfr{u} + \frac{L}{V\,\Gamma^p}\,\left(\frac{R}{2}\right)^{2s}\,\left(\frac{A}{2}\right)^{p+1}\,(p+1)^{\tau} \right]\leq 1,
	\eeq
for all $p$. It is clear that this is always the case, more precisely, since $\Gamma$ appears always at the denominator with the highest exponent in $p$, it is possible to choose $\Gamma$ (depending on $E$, $F$, $\linfr{u}$, and $A$) so that (\ref{N4}) holds for all $V\geq 1$.
\bigskip

It is well known that for $s\in(0,1)$ the fractional Laplacian $\Laps$ is the translation invariant integro-differential operator whose kernel, denoted by $K_0$, is defined as
	\beqs
	K_0(y)=-\frac{1}{2}c_{n,s}\frac{1}{|y|^{n+2s}}
	\eeqs
here $c_{n,s}$ denotes a normalization constant, see for instance \cite{DNPV} for a survey on the topic.
 
We observe that the kernel $K_0$ is analytic out of the origin since it is a composition of analytic functions (see for instance \cite{KP} Proposition 1.6.7). Furthermore it is a homogeneous function of degree $-(n+2s)$ and this implies that for each multi-index $\alpha\in\N^n$, the partial derivative $D^{\alpha}K_0$ is a homogeneous function of degree $-(|\alpha|+n+2s)$.\\
Analiticity of $K_0$ implies that there exist positive constants $R$ and $C$ such that for any $y$ with $|y|=1$ and for any $\alpha\in\N^n$ we have that
	\beqs
	\left|D^{\alpha}K_0(y)\right|\leq C\frac{j!}{R^j},
	\eeqs
where $j=|\alpha|$, (see, e.g. \cite{KP}).\\
From the homogeneity of the kernel we can conclude that conditions (\ref{K1}), (\ref{K2}), and (\ref{K3}) hold with $\nu=1$ and Corollary \ref{B} follows.

\section*{Acknowledgements} We thank Matteo Cozzi and Bego\~na Barrios for their careful reading of a preliminary version of this paper and Francesco Monopoli for several interesting conversations about the last section of the paper. The authors thank the anonimous referees for their useful comments and remarks. This paper is the outcome of a class project of a course held at the Universit\`a degli Studi di Milano in 2013.\\

\end{document}